\documentclass[11pt]{article}

\usepackage{amssymb}
\usepackage{amsmath}
\usepackage{amsfonts}
\usepackage{enumerate}

\usepackage{tikz}
\usetikzlibrary{patterns,snakes}
\usetikzlibrary{shapes.geometric}
\usepackage[makeroom]{cancel}
\usepackage{amsthm}

\usepackage{graphicx}

\newcommand{\bN}{ {\mathbb{N} } }

\newcommand{\bR}{ {\mathbb{R}} }
\newcommand{\bZ}{ {\mathbb{Z}} }
\newcommand{\bC}{ {\mathbb{C}} }

\newcommand{\bE}{ {\mathbb{E}} }
\newcommand{\stirling}[2]{\genfrac{[}{]}{0pt}{}{#1}{#2}}

\newtheorem*{theorem}{Theorem}
\newtheorem*{remark}{Remark}

\newtheorem*{lemma}{Lemma}
\newtheorem*{proposition}{Proposition}
\newtheorem*{corollary}{Corollary}
\newtheorem*{definition}{Definition}

\usepackage{titling}

\pretitle{\begin{center}\Huge\bfseries}
\posttitle{\par\end{center}\vskip 0.5em}
\preauthor{\begin{center}\Large\ttfamily}
\postauthor{\end{center}}
\predate{\par\large\centering}
\postdate{\par}

\title{Combinatorial Insights into Asymptotic Normality}
\author{Alexandru Gatea, Simon Huang}
\date{\today}

\begin{document}

\maketitle
\thispagestyle{empty}

\begin{abstract}
In this paper we examine various sequences $(X_n)_{n \in \bN}$ of real random variables arising naturally in connection with certain combinatorial structures. The goal is to investigate sufficient conditions for convergence to the standard normal distribution upon centering and normalization.
By studying signed binomial sums, Stirling numbers of both kinds, and factorial moments, we obtain certain cancellations in the computation of centred moments $\bE[(X_n - \bE[X_n])^k]$. This allows us to conclude that $\frac{X_n - \bE[X_n]}{\sqrt{Var  \ X_n}} \xrightarrow{d} N(0,1)$ when the factorial moments $\bE[(X_n)_k]$ are rational functions satisfying some simple conditions on the coefficients. A particular class of such random variables are those whose factorial moments are products and quotients of falling factorials in $n$; we apply our result in this situation and we discuss specific examples arising from block statistics of $k$-divisible non-crossing partitions and from outdegree statistics of rooted plane trees with $k$-divisible outdegrees.
\end{abstract}

\begin{center}

\section{Introduction}
\end{center}

A very important and frequently occurring notion in probability theory is that of asymptotic normality of a sequence of random variables. For instance, the Central Limit Theorem has numerous versions, involving different situations and different modes of convergence. In this paper we will study sequences of real random variables that converge in distribution to the standard normal distribution. Our result will be most (but not solely) applicable to sequences of discrete random variables taking values in $\bZ^{\geq 0}$.\\

\noindent More concretely, we seek conditions on sequences of real random variables $(X_n)_{n \in \bN}$ that imply $\frac{X_n - \bE[X_n]}{\sqrt{Var(X_n)}}  \xrightarrow{d} N(0,1)$, where $Var(X_n) = \bE[(X_n - \bE[X_n])^2]$. This problem has been studied quite extensively, with a plethora of results found in the literature. One frequently used tool is the Janson Dependency Criterion (see \cite{janson}). For instance, in \cite{bona}, B\'{o}na used it to prove that for a fixed $k \in \bN$, the normalized random variable corresponding to the number of occurences of a $12 \cdots k$ pattern in a random $n$-permutation (as $n \rightarrow \infty$) is asymptotically normal. In \cite{cumulants}, Bobecka et al. used it to give sufficient asymptotic conditions on factorial cumulants that ensure asymptotic normality, and derived a general technique which they applied to various examples. However, both approaches required rather involved manipulations of the random variables and/or their moments to show that the theorems' conditions were indeed satisfied. The prevalent issue in the strategies for asserting asymptotic normality presented in these papers (and others) is that the key step often requires potentially unwieldy and unintuitive calculations to verify the conditions stated in the main theorem. The conditions derived by Gao and Wormald \cite{gao} are simpler, but still require the factorial moments to be asymptotic to some exponential expressions. When the factorial moments are rational (as is the case for many simple examples), a certain amount of manipulation is required before the results may be applied. \\

\noindent In this paper, we take an elementary combinatorial approach to give conditions on the factorial moments (involving extraction of coefficients from polynomials) that imply asymptotic normality, and are easily verifiable in the case where they are rational functions. Our approach follows in the same spirit as Weiss, generalizing a technique used in \cite{weiss}. When expanding the centred moments in terms of the (uncentred) factorial moments, we demonstrate cancellations in the high-degree terms. The simplicity of our result comes at the cost of some flexibility: our conditions are stricter (they require the factorial moments to be of a precise form, rather than asymptotic equality), and they require the factorial moments to be functions of one variable (whereas other similar limit laws allow for additional variables, so long as they are all asymptotically proportional). \\

\noindent After a preliminary section outlining elementary facts about signed binomial sums, the difference operator, Stirling numbers of both kinds and convergence in distribution, we proceed to our results. We first prove a general theorem -- which requires checking a somewhat unpleasant condition in order to be applied -- and proceed to give the following theorem with an alternative condition, which is simpler and relatively easy to apply.\\

\begin{theorem}[3.2.1]
Let $w_1, w_2 \in \bC, \delta \in \bC \setminus \{0\}$ be constants with $w_1 \neq w_2$ (1) and with \\ 
$2 \delta (w_1-w_2) \neq -1$ (1). Let $(q_p)_{p \geq 0}, (r_p)_{p \geq 0}$ be nonnegative integers satisfying 
$$q_0 = r_0 = 0 \ \mbox{and} \ q_p - r_p = p \ \forall p \in \bN \ (2)$$
and for every $p \geq 0$ let $f_p(x), g_p(x)$ be monic complex polynomials of degrees $q_p, r_p$ respectively.

For every $m \geq 0, p \geq 0$ define $\zeta_m(p) := [x^{q_p - m}] f_p(x)$ and $\tau_m(p):= [x^{r_p-m}] g_p(x)$ (the $m\textsuperscript{th}$ highest coefficients of $f_p(x), g_p(x)$). Suppose $\zeta_m(p), \tau_m(p)$ are polynomials in $p$ of degree $2m$ with respective leading coefficients $\frac{w_1^m}{m!}, \frac{w_2^m}{m!} \ \forall m, p \geq 0$ (3). 

Let $(\chi_n)_{n \in \bN}$ be complex numbers such that $\lim \limits_{n \rightarrow \infty} \chi_n = \infty$ (4). If $(X_n)_{n \in \bN}$ are real random variables satisfying
$$\bE [(X_n)_p] =  \frac{ \delta^p f_p(\chi_n)}{g_p(\chi_n)}  \ \forall p \in \bN, n \in \bN, n > p \ (5)$$
then $\frac{X_n - \bE[X_n]}{\sqrt{Var(X_n)}}  \xrightarrow{d} N(0,1)$.
\end{theorem}

\noindent Using this result, we obtain a very useful corollary for the case where the $f_p(x), g_p(x)$ are of a ``nice enough form'' involving products of falling factorials (see Theorem 3.3.1). To get a better idea of the practicality of this theorem, one can compare it with the general result in \cite{gao} involving factorial moments; Gao and Wormald's theorem undoubtedly covers a wider range of examples (because one only needs to deal with the asymptotics of the factorial moments), but in situations where the factorial moments can be computed exactly and are of this ``nice form" (which is quite often), it is more convenient to apply our result.\\

\noindent We subsequently study some examples where Theorem 3.3.1 applies. We begin with a quick verification for some simple discrete distributions (e.g. binomial, negative binomial). We proceed to highlight a few more interesting examples arising from block statistics of $k$-divisible non-crossing partitions, such as the number of blocks of size $tk$ in a randomly selected partition $\pi \in NC^k(n)$ (see section 4.3). By looking at generating functions, these examples translate to analogous results in the context of rooted plane trees, a related structure also counted by the Catalan numbers (e.g. the number of vertices with outdegree $tk$ in a randomly chosen rooted plane tree $T \in T^k(n)$ is also asymptotically normal). For example, under this bijection, the following is a rooted plane tree and its corresponding non-crossing partition:\\
\begin{center}
\begin{tikzpicture}
[level distance=10mm,
 my circle/.style={draw, fill, circle, minimum size=5pt, inner sep=0pt},
level 1/.style={sibling distance=30mm},
level 2/.style={sibling distance=15mm},
level 3/.style={sibling distance=10mm}]
\node[my circle]  {} [grow'=up]
child {node[my circle] {}
child {node[my circle] {}
child {node[my circle] {}}
child {node[my circle] {}}
}
child {node[my circle] {}}
child {node[my circle] {}
child {node[my circle] {}}
child {node[my circle] {}}
}
}
child {node[my circle] {}
child {node[my circle] {}
child[missing]
child {node[my circle] {}}
}
child {node[my circle] {}}};
\end{tikzpicture}\\
\end{center}
\begin{center}
\begin{tikzpicture}[thick,font=\small]
     \path (0,0) node (a) {1}
           (1,0) node (b) {2}
           (2,0) node (c) {3}
           (3,0) node (d) {4}
           (4,0) node (e) {5}
           (5,0) node (f) {6}
           (6,0) node (g) {7}
           (7,0) node (h) {8}
           (8,0) node (i) {9}
           (9,0) node (j) {10}
           (10,0) node (k) {11}
           (11,0) node (l) {12};
     \draw (a) -- +(0,0.75) -| (i);
     \draw (j) -- +(0,0.75) -| (l);
     \draw (k) -- +(0,0.5) ;
     \draw (b) -- +(0,0.6) -| (e);
     \draw (e) -- +(0,0.6) -| (f);
     \draw (g) -- +(0,0.6) -| (h);
     \draw (c) -- +(0,0.4) -| (d);

   \end{tikzpicture}
\end{center}
Through these examples, we show how factorial moments can be computed efficiently using Lagrange's Implicit Function Theorem, and how our result can be applied easily to assert asymptotic normality.

\begin{center}

\section{Preliminaries}
\end{center}

\subsection{The Binomial Coefficient}
We recall the definition of and some useful facts about the binomial coefficient.
\begin{definition} [2.1.1]
The \underline{binomial coefficient} ${j \choose k}$ is defined by ${j \choose k} = \begin{cases} \frac{j!}{k!(j-k)!} &  0 \leq k \leq j \\ 0 & k < 0 \leq j \\ (-1)^k {-j + k - 1 \choose k} & j < 0 \leq k \end{cases}$.
\end{definition}

\noindent \begin{remark} [2.1.2]
The following are elementary, well-known results.
\begin{itemize}
\item ${j \choose k} = {j \choose j-k}$ for all $j \in \bZ^{\geq 0}, k \in \bZ$
\item ${j \choose k} = {j-1 \choose k} + {j-1 \choose k -1 }$ for all $j \in \bN, k \in \bZ$
\item ${j \choose k}$ is a polynomial in $j$ of degree $k$ for all $k \in \bZ^{\geq 0}$
\item (The Binomial Theorem) $(x+y)^p = \sum \limits_{j=0}^{\infty} {p \choose j} x^j y^{p-j}$ for all $x, y \in \bR, p \in \bZ$ (for $p > 0$, this sum runs only for $0 \leq j \leq p$ since the rest of the terms are 0)
\end{itemize}
\end{remark}

\subsection{The Difference Operator}
Next, we look at the theory of difference operators. The results discussed in this section, based on \cite{scheinerman}, will allow us to express Stirling numbers of both kinds as polynomials of a particular degree and leading coefficient in the next section, without computing those polynomials explicitly.

\begin{definition}[2.2.1]
For a sequence of real numbers $(a_n)_{n \in \bN}$, we define the difference operator to be $\Delta a_n := a_{n+1} - a_n \ \forall n \geq 0$. 
\end{definition}

\noindent If $a_n$ is a polynomial in $n$ then the difference operator has some nice properties.

\begin{proposition}[2.2.2]
Consider a sequence $(a_n)_{n \in \bN}$ which is a real polynomial in $n$ of degree $d \geq 1$ with leading coefficient $\alpha \neq 0$, say $a_n = \alpha n^d  + \beta n^{d-1} +\rho(n)$ where $\rho(x)$ is a polynomial of degree at most $d-2$. Then $\Delta a_n$ is a polynomial in $n$ of degree $d-1$ with leading coefficient $\alpha d$. Furthermore, $\Delta^d a_n = \alpha d! \ \forall n \in \bN$ and $\Delta^{d+1} a_n = 0 \ \forall n \in \bN$.
\end{proposition}

\begin{proof} We have
$$\Delta a_n = \alpha [(n+1)^d - n^d] + \beta[(n+1)^{d-1} - n^{d-1}] + \rho(n+1) - \rho(n) = \alpha dn^{d-1} + \tilde{\rho}(n) $$
for some polynomial $\tilde{\rho}(x)$ of degree at most $d-2$. Thus $\Delta a_n$ is a polynomial in $n$ of degree $d-1$ with leading coefficient $\alpha d$. By induction, it follows that $\Delta^{d} a_n$ is a polynomial in $n$ of degree 0 with leading coefficient $\alpha d (d-1) \cdots 1 = \alpha d!$ (i.e. $\Delta^{d} a_n = \alpha d!$) and $\Delta^{d+1} a_n = 0$ for all $n \in \bN$. \\
\end{proof} 

\noindent With some work, we can prove the converse. In particular, the overall goal of this section is to obtain a sufficient condition on the difference operator of a sequence of numbers $(a_n)_{n \in \bN}$ that will imply that $a_n$ is a polynomial in $n$ of degree $d$ with leading coefficient $\alpha$, for some $d$ and $\alpha$. The following lemma, proved by Scheinerman in \cite{scheinerman}, shows that the difference operator uniquely determines a sequence of numbers.

\begin{lemma}[2.2.3]
Let $(a_n)_{n \geq 0}, (b_n)_{n \geq 0}$ be sequences of numbers and let $d \geq 1$. Assume $a_0= b_0$, $\Delta^d a_n = \Delta^d b_n = 0 \ \forall n \geq 0$ and $\Delta^j a_0 = \Delta^j b_0 \ \forall 1 \leq j < d$. Then $a_n = b_n \ \forall n \in \bN$.
\end {lemma}

%

\noindent Using Lemma 2.2.3, Scheinerman \cite{scheinerman} characterizes precisely when such an $a_n$ can be expressed as a polynomial in $n$, and what this polynomial will be in terms of the difference operator.

\begin{proposition}[2.2.4]
Let $(a_n)_{n \geq 0}$ be a sequence of numbers. Then $a_n$ can be expressed as a polynomial in $n$ if and only if $\exists \ k \geq 0$ s.t. $\Delta^{k+1} a_n = 0 \ \forall n \geq 0$. If this is the case, then  
$$a_n = a_0 {n \choose 0} + (\Delta a_0) {n \choose 1} + \ldots + (\Delta^k a_0) {n \choose k}.$$
\end {proposition}

\noindent Finally, by looking at the degree and leading coefficient of the polynomial expression in Proposition 2.2.4, we get a necessary and sufficient condition on $\Delta a_n$ that implies $a_n$ is a polynomial in $n$.

\begin{corollary}[2.2.5]
Let $(a_n)_{n \geq 0}$ be a sequence of numbers, let $d \geq 1$ and let $\alpha \in \bC \setminus \{0\}$. Then $a_n$ is a polynomial of degree $d$ with leading coefficient $\alpha$ if and only if $\Delta a_n$  is a polynomial of degree $d-1$ with leading coefficient $\alpha d$. 
\end {corollary}
\begin{proof}
The $\Rightarrow$ direction was already proved in Proposition 2.2.2. Suppose $\Delta a_n$  is a polynomial of degree $d-1$ with leading coefficient $\alpha d$. Then $\Delta^{d+1} a_n = \Delta^{d-1+1} (\Delta a_n) = 0$
$$\mbox{so} \ a_n = a_0 {n \choose 0} + (\Delta a_0) {n \choose 1} + \ldots + (\Delta^d a_0) {n \choose d} \ \mbox{by Proposition 2.2.4}$$
and thus $a_n$ is a polynomial in $n$ of degree $d$ with leading coefficient $\frac{\Delta^d a_0}{d!}$ by Remark 2.1.2. Since $\Delta a_n$ is a polynomial of degree $d-1$, then by Proposition 2.2.2 we may see that $\Delta^d a_n = \Delta^{d-1} (\Delta a_n) = (d-1)! (\alpha d) = d! \alpha$, so $a_n$ is a polynomial of degree $d$ with leading coefficient $\alpha$.\\
\end{proof}

\subsection{Stirling Numbers of the First and Second Kinds}
Stirling numbers of both kinds occur throughout combinatorics. The unsigned Stirling number of the first kind $\stirling{j}{k}$ counts the number of permutations in $S_j$ with $k$ disjoint cycles. The Stirling number of the second kind ${j \brace k}$ counts the number of partitions a set of size $j$ into $k$ non-empty subsets.

The Stirling numbers also appear when passing between powers and falling factorials. Our general result is stated with conditions on the factorial moments $\bE[(X_n)_k]$ (because these are easier to compute in practice), so Stirling numbers of the second kind will appear as the coefficients when writing the raw moments $\bE[X_n^k]$ in terms of the factorial moments. In a later section, when we investigate what sort of rational functions fit our general result, we will go in the other direction: a falling factorial $(x)_k$ can be written as a polynomial in $x$, with Stirling numbers of the first kind appearing as the coefficients.

\begin{definition} [2.3.1]
 For $j \in \bZ^{\geq 0},k \in \bZ$ define the \underline{unsigned Stirling number of the first kind} $\stirling{j}{k}$ to be determined by $(x)_j = \sum \limits_{k=0}^j \stirling{j}{k} (-1)^{j-k} x^k \ \forall j \in \bZ^{\geq 0}$ and by $\stirling{j}{k} = 0$ for $k > j$ or $k < 0$. 
\end{definition}

The following recurrence relation for unsigned Stirling numbers of the first kind follows immediately from Definition 2.3.1.

\begin{proposition}[2.3.2] 
$\stirling{j}{k} = \stirling{j-1}{k-1} + (j-1) \stirling{j-1}{k} \ \forall j \in \bN, k \in \bZ$
\end{proposition}

\begin{definition} [2.3.3]
 For any $j \in \bZ^{\geq 0},k \in \bZ$ define the \underline{Stirling number of the second kind} ${j \brace k}$ to be determined by $x^j = \sum \limits_{k=0}^j {j \brace k} (x)_k \ \forall j \in \bZ^{\geq 0}$ and by ${j \brace k} = 0$ for $k > j$ or $k < 0$. 
\end{definition}

The following recurrence relation for Stirling numbers of the second kind follows immediately from Definition 2.3.3.

\begin{proposition}[2.3.4]
${j \brace k} = {j-1 \brace k-1} + k {j-1 \brace k} \ \forall j \in \bN, k \in \bZ$
\end{proposition}

\noindent We now apply Corollary 2.2.5 from the previous section in order to express Stirling numbers of the forms $\stirling{j}{j-i}$ and ${j \brace j-i}$ as polynomials in $j$ of a certain degree and leading coefficient. This is the essential result from this section on Stirling numbers.

\begin{proposition}[2.3.5]
For every fixed $i \geq 0$, $\stirling{j} {j-i}$ and ${j \brace j-i}$ are polynomials in $j$ of degree $2i$ with leading coefficient $\frac1{2^i i!}$.
\end {proposition}

\begin{proof}
Fix any $i \geq 0$ and define $\rho_i(j) =\stirling{j} {j-i}, \eta_i(j) = {j \brace j-i} \ \forall j \geq 0$. We will prove that $\rho_i(j)$ and $\eta_i(j)$ are polynomials in $j$ of degree $2i$ with leading coefficient $\frac1{2^i i!}$ for all $i \geq 0$ by induction on $i$.\\
\underline{Base Case:} $i=0$\\
We have $\rho_0(j) = \stirling{j} {j} = 1 = \frac1{2^0 0!} j^0 = {j \brace j} = \eta_0(j)$ so the claim holds for $i=0$.\\
\underline{Inductive Hypothesis:} Suppose $\rho_i(j), \eta_i(j)$ are polynomials in $j$ of degree $2i$ with leading coefficient $\frac1{2^i i!}$ for some $i \geq 0$.\\
\underline{Inductive Step:} Define the sequences $a_j := \rho_{i+1}(j), b_j := \eta_{i+1}(j)  \ \forall j \geq 0$. Applying Propositions 2.3.2 and 2.3.4 we get the recurrence relations
$$\stirling{j+1} {j-i}=\stirling{j+1} {j+1-i-1} = \stirling{j} {j-i-1} + j \stirling{j} {j-i} \ \forall j \in \bZ^{\geq 0}$$
$${j+1 \brace j-i}= {j+1 \brace j+1-i-1} = {j \brace j-i-1} + (j-i) {j \brace j-i} \ \forall j \in \bZ^{\geq 0}$$
which are equivalent to $a_{j+1} = a_j + j \rho_{i}(j), b_{j+1} = b_j + (j-i) \eta_{i}(j) \ \forall j \in \bZ^{\geq 0} $. That is,
$$\Delta a_j = j \rho_i(j) , \Delta b_j = (j-i) \eta_i(j) \ \forall j \in \bZ^{\geq 0}$$
By the Induction Hypothesis, both $ j \rho_i(j)$ and $ (j-i) \eta_i(j)$ are polynomials in $j$ of degree $2i+1$ with leading coefficient $\frac1{2^i i!}$. By Corollary 2.2.5, we get that $a_j, b_j$ are polynomials in $j$ of degree $2i+2$ with leading coefficient $\frac1{2i+2} \cdot \frac1{2^i i!} = \frac1{2^{i+1} (i+1)!}$.\\
\end{proof}

\subsection{Signed Binomial Sum Identities}
The key idea in the approach we will use to prove asymptotic normality is that certain signed binomial sums are very easy to simplify. In this small section, we will give two identities that faciliate the calculation of the leading term of $\bE[(X_n - \bE[X_n])^k]$ in the proof of the general result. 

\begin{lemma}[2.4.1]
Let $p \in \bN, m \in \bZ^{\geq 0}$ with $ m \leq p$. Then
$$\sum_{j=0}^p {p \choose j} (-1)^{p-j} (j)_m = \begin{cases} p! & \mbox{if} \  m = p \\ 0 & \mbox{if} \ m < p \end{cases} $$
where $(j)_m$ is the falling factorial $j (j-1) \cdots (j-m+1)$ with $(j)_0 := 1$.
\end {lemma}

\begin{proof}
Observe that
\begin{align*}
\sum_{j=0}^p {p \choose j} (-1)^{p-j} (j)_m &= \sum_{j=m}^p \frac{p!}{j! (p-j)!} \cdot j (j-1) \cdots (j-m+1) (-1)^{p-j}  \\
&\quad \text{since the terms in the sum with $0 \leq j \leq m-1$ are 0}\\
&= \sum_{j=m}^p p(p-1) \cdots (p-m+1) \frac{(p-m)!}{(j-m)! (p-j)!}  (-1)^{p-j}  \\
&= (p)_m \sum_{j=m}^p {p-m \choose j-m}  (-1)^{(p-m) -(j-m)}   \\
&= \begin{cases} (p)_p & \mbox{if} \ m = p \\ (p)_m (-1+1)^{p-m} & \mbox{if} \ m < p \end{cases} \\
&=  \begin{cases} p! & \mbox{if} \  m = p \\ 0 & \mbox{if} \ m < p \end{cases} 
\end{align*}
by the Binomial Theorem.\\
\end{proof}

\noindent This lemma itself is quite a useful identity. However, in the proof of the general result, we will have signed binomial sums with some unknown polynomial in $j$ in the place of the falling factorial $(j)_m$. Consequently, we will work with the following more general identity, which is a straightforward corollary of Lemma 2.4.1 and Definition 2.3.3.
\begin{proposition}[2.4.2]
Let $p \in \bN$ and let $f(x)$ be a polynomial of degree at most $p-1$. Then
$$\sum_{j=0}^p {p \choose j} (-1)^{p-j} j^p = p! \ \mbox{and} \ \sum_{j=0}^p {p \choose j} (-1)^{p-j} f(j) = 0 .$$
\end {proposition}

\subsection{Extracting Coefficients}
The proof of the general result relies on extracting coefficients from various polynomials. The following result will aid us in this calculation.

\begin{proposition}[2.5.1]
Let $f(x)$ be a complex polynomial of degree $d \geq 1$ and let $m \geq 0$. If $a_j := [x^{jd - m}] f(x)^j \ \forall j \geq 0$ then $a_j$ is a complex polynomial in $j$ of degree at most $m$.
\end{proposition}
\begin{proof}
Since $f(x)$ is a nonconstant complex polynomial of degree $d$, then by the Fundamental Theorem of Algebra it can be factored as $f(x) = \delta \prod \limits_{s=1}^d (x+ \alpha_s)$ for some (not necessarily distinct) $\alpha_s \in \bC$ and leading coefficient $\delta \in \bC \setminus \{0\}$. But then $\forall j \geq 0$, we have:
$$a_j = \delta [x^{jd - m}] \prod \limits_{s=1}^d (x +\alpha_s)^j = \delta \sum_{\substack{m_1, \ldots, m_d \geq 0\\ m_1 + \ldots + m_d = m}} \prod \limits_{s=1}^d [x^{j - m_s}] (x +  \alpha_s)^j = \delta \sum_{\substack{m_1, \ldots, m_d \geq 0\\ m_1 + \ldots + m_d = m}} \prod \limits_{s=1}^d {j \choose j - m_s} \alpha_s^{m_s}$$
by the Binomial Theorem, where we sum over all tuples $(m_1,\ldots,m_d)$, including those with some $m_s > j$ (since the corresponding binomial coefficient will be 0). By Remark 2.1.2,  every ${j \choose j - m_s}$ is a polynomial in $j$ of degree $m_s$. Thus each $\prod \limits_{s=1}^d {j \choose j - m_s} \alpha_s^{m_s}$ is a polynomial in $j$ of degree $\sum \limits_{s=1}^d m_s = m$, so $a_j$ is a polynomial in $j$ of degree at most $m$. 

\end{proof}

\subsection{Convergence in Distribution to $N(0,1)$}
Before discussing the main results of this paper, we must briefly recall the definition of convergence in distribution, and also give a sufficient condition for asymptotic normality. 

\begin{definition}[2.6.1]
A sequence of random variables $(X_n)_{n \in \bN}$ with respective cumulative distribution functions $(F_n)_{n \in \bN}$ \underline{converges in distribution} to a random variable $X$ with cumulative distribution function $F$ if $ \lim \limits_{n \rightarrow \infty} F_n(t)  = F(t) \ \forall t$ where $F$ is continuous. We denote this by $X_n \xrightarrow{d} X$.
\end{definition} 

\begin{definition}[2.6.2]
A distribution $\mu$ is \underline{uniquely determined by its moments} if whenever a distribution $\nu$ has moments of all orders equal to the corresponding moments of $\mu$, we have $\mu = \nu$.
\end{definition} 

\noindent The following known result, proven by Billingsley in \cite{billingsley}, states that if a random variable $X$ is uniquely determined by its moments $\bE[X^k]$, then a sequence of random variables $(X_n)_{n \in\bN}$ converges to $X$ in distribution if each sequence of moments converges to the corresponding moment of $X$.

\begin{proposition}[2.6.3]
Let $(X_n)_{n \in \bN}$ and $X$ be real random variables with finite moments of all orders, such that the distribution of $X$ is uniquely determined by its moments. If $ \lim \limits_{n \rightarrow \infty} \bE[X_n^k]  = \bE[X^k] \ \forall k \in \bN$ then $X_n \xrightarrow{d} X$.
\end{proposition} 

\noindent In \cite{billingsley}, Billingsley also shows that the standard normal distribution is uniquely determined by its moments. Moreover, all higher-order moments of the standard normal distribution are well-known.

\begin{proposition}[2.6.4]
Let $X \sim N(0,1)$. Then $\bE[X^k] = \begin{cases} \frac{k!}{(k/2)! 2^k} & \mbox{$k$ is even} \\ 0 & \mbox{$k$ is odd} \end{cases}$ for all $k \in \bN$.
\end{proposition}

\noindent As discussed before, our goal is to find sufficient conditions on real random variables $(X_n)_{n \in \bN}$ that imply that $\frac{X_n - \bE[X_n]}{\sqrt{Var(X_n)}}  \xrightarrow{d} N(0,1)$. By the results in this section, it suffices to seek conditions that imply 
$$\lim_{n \rightarrow \infty} \frac{\bE[(X_n - \bE[X_n])^p]}{Var(X_n)^{p/2}} = \begin{cases} \frac{p!}{(p/2)! 2^p} &   \mbox{$p$ is even} \\ 0 & \mbox{$p$ is odd} \end{cases}  \ \forall p \in \bN.$$

\begin{center}

\section{The General Results}
\end{center}
\subsection{The Main Theorem}
\noindent The following is the main result obtained in this paper; it states that if $(X_n)_{n\in\bN}$ is a sequence of real random variables whose factorial moments are rational expressions in some complex $\chi_n$ (where $\chi_n \rightarrow \infty$) with coefficients and degrees satisfying certain conditions, then $\frac{X_n - \bE[X_n]}{\sqrt{Var(X_n)}}  \xrightarrow{d} N(0,1)$.

\begin{theorem}[3.1.1]
Let $w, \delta \in \bC \setminus \{0\}$ be constants with $2 \delta w \neq -1$ (1). Let $(q_p)_{p \geq 0}, (r_p)_{p \geq 0}$ be nonnegative integers satisfying 
$$q_0 = r_0 = 0 \ \mbox{and} \ q_p - r_p = p \ \forall p \in \bZ^{\geq 0} \ (2).$$

For every $p \geq 0$ let $f_p(x), g_p(x)$ be monic complex polynomials of degrees $q_p, r_p$ respectively. For every $m \geq 0, p \geq 1, 0 \leq j \leq p$ define $a_{m,p}(j):= [x^{q_j - r_j + \sum \limits_{h=1}^p r_h-m} ]f_j(x) \prod \limits_{1 \leq h \leq p, h \neq j} g_h(x)$ (the $m\textsuperscript{th}$ highest coefficient of the expression). Suppose that $a_{m,p}(j)$ is a polynomial in $j$ of degree $2m$ with leading coefficient $\frac{w^m}{m!}$ (3), for all $m \geq 0, p \geq 1$.

Let $(\chi_n)_{n \in \bN}$ be a sequence of complex numbers such that $\lim \limits_{n \rightarrow \infty} \chi_n = \infty$ (4). If $(X_n)_{n \in \bN}$ is a sequence of real random variables satisfying
$$\bE [(X_n)_p] =  \frac{ \delta^p f_p(\chi_n)}{g_p(\chi_n)}  \ \forall p \in \bN, n \in \bN, n > p \ (5)$$
$$\mbox{then} \ \lim_{n \rightarrow \infty} \frac{\bE [(X_n - \bE [X_n])^p]} {\bE [(X_n - \bE [X_n])^2]^{p/2}}= 
\begin{cases}
\frac{p!}{(p/2)! 2^{p/2}} & \mbox{$p$ is even} \\
0 & \mbox{$p$ is odd} \\
\end{cases} \ \mbox{so} \ \frac{X_n - \bE[X_n]}{\sqrt{Var(X_n)}}  \xrightarrow{d} N(0,1).$$

\end{theorem}

\begin{proof} Substituting $p = 1$ in (5) yields the mean $\bE[X_n] =  \frac{ \delta f_1(\chi_n)} {g_1(\chi_n)}$. Moreover, $f_0 = g_0 = 1$ so $\bE [(X_n)_p] =  \frac{ \delta^p f_p(\chi_n)}{g_p(\chi_n)}  \ \forall p \geq 0, n \in \bN$. Now fix any even $p \in \bN$ and any $n \in \bN$ with $n > p$.\\

\noindent \underline{Step 1: Expand $\bE [(X_n - \bE [X_n])^p] $}\\
Recall the definition of the Stirling numbers of the second kind. By linearity of expectation, we may write the raw moments as linear combinations of factorial moments: 
$$\bE[X_n^j] =  \sum \limits_{i=0}^j {j \brace j-i} \bE[(X_n)_{j-i}] \ \forall j \geq 0$$
We then expand the centered moments using the Binomial Theorem.
\begin{align*}
\bE [(X_n - \bE [X_n])^p] &= \sum \limits_{j = 0}^p {p \choose j} (-1)^{p-j} \bE[X_n]^{p-j} \sum \limits_{i=0}^j {j \brace j-i} \bE[(X_n)_{j-i}] \\
&= \frac{\sum \limits_{j = 0}^p {p \choose j} (-1)^{p-j} \sum \limits_{i=0}^j {j \brace j-i} \delta^{p-j} f_1(\chi_n)^{p-j} g_1(\chi_n)^j \cdot  \delta^{j-i} f_{j-i} (\chi_n) \prod \limits_{1 \leq h \leq p, h \neq j-i} g_h(\chi_n)  }{ g_1(\chi_n)^p \prod \limits_{1 \leq h \leq p} g_h(\chi_n) }\\
&:= \frac{\sum \limits_{j=0}^p { p \choose j} (-1)^{p-j} \sum \limits_{i=0}^j {j \brace j-i} C_{i,j} (n,p)}{D(n,p)} 
\end{align*}
where $D(n,p)$ and all $C_{i,j}(n,p)$ are clearly polynomials in $\chi_n$. In particular, the degree of $\chi_n$ in $D(n,p)$ is $e(p) := \deg_{\chi_n} D(n,p) = p r_1 + \sum \limits_{h=1}^p r_h$ and the degree of $\chi_n$ in $C_{i,j}(n,p)$ for any $0 \leq j \leq p, 0 \leq i \leq j$ is $d_{i}(p) := \deg_{\chi_n} C_{i,j}(n,p) = (p-i) + e(p)$ by a straightforward computation using (2). Observe that $\deg_{\chi_n} C_{i,j} (n,p) $ depends only on $i$ and $p$ and not on $j$.\\

\noindent \underline{Step 2: Compute the coefficient of $\chi_n^{d_{p/2}(p)}$ in $\sum \limits_{i=0}^j {j \brace j-i} C_{i,j} (n,p)$ for each fixed $0 \leq j \leq p$}\\
Fix $0 \leq j \leq p$.  If $i > j$ then $ {j \brace j-i} = 0$ and if $i > p/2$ then 
$$d_{i}(p) =  (p-i) + e(p) <  (p - p/2)+ e(p) = d_{p/2}(p) .$$
We may then see that $[\chi_n^{d_{p/2}(p)}] \sum \limits_{i=0}^j {j \brace j-i} C_{i,j} (n,p) = \sum \limits_{i=0}^{p/2} {j \brace j-i} [\chi_n^{d_{p/2}(p)}]  C_{i,j} (n,p)$, so the bounds of our sum need not depend on $j$. We will compute this coefficient as a polynomial in $j$ so that we can apply the signed binomial sum identity to get nice cancellations. \\

\noindent Recall that $ {j \brace j-i} = \frac{j^{2i}}{i! 2^i} +\rho_i(j)$ for some polynomial $\rho_i(x)$ of degree at most $2i-1$ from Proposition 2.3.5. We will compute $[\chi_n^{d_{p/2}(p)}] C_{i,j} (n,p)$ for each $0 \leq i \leq p/2$ as a polynomial in $j$ with coefficients depending on $i$ and $p$. By Proposition 2.4.2, we are only concerned with the coefficients of $j^m$ for $m \geq p-2i $ (*).\\

\noindent Fix any $0 \leq i \leq p/2$.  Observe that 
$$C_{i,j}(n,p) = \delta^{p-i} f_1(\chi_n)^{p-j} g_1(\chi_n)^j  f_{j-i} (\chi_n) \prod \limits_{1 \leq h \leq p, h \neq j-i} g_h(\chi_n)$$
has degree $d_{i}(p) = d_{p/2}(p)+p/2 - i$ in $\chi_n$. By definition of the $a_{m,p}(x)$, we have:
$$[\chi_n^{d_i(p) - p/2 + i}] C_{i,j}(n,p) = \delta^{p-i} \sum_{\substack{s_1 + s_2 + s_3 = \frac{p}2 - i \\ s_1,s_2,s_3 \geq 0}} [x^{(p-j)q_1-s_1}] f_1(x)^{p-j} \cdot [x^{j r_1-s_2}] g_1(x)^{j} \cdot  a_{s_3,p}(j-i) $$
(i.e. we ``lose" a factor of $\chi_n^{s_1}$ from $f_1(\chi_n)^{p-j}$, a factor of $\chi_n^{s_2}$ from $g_1(\chi_n)^{j}$, and a factor of $\chi_n^{s_3}$ from $f_{j-i} (\chi_n) \prod \limits_{1 \leq h \leq p, h \neq j-i} g_h(\chi_n)$, for all nonnegative tuples $(s_1,s_2,s_3$)). Note that we may include tuples $(s_1,s_2,s_3)$ that are outside the range of the corresponding degrees. \\

Consider the degree of $j$ in each summand, fixing any $s_1,s_2,s_3$. By Proposition 2.5.1, $[x^{(p-j)q_1-s_1}] f_1(x)^{p-j}$ is polynomial of degree at most $s_1$ in $p-j$ (so it is a polynomial of degree at most $s_1$ in $j$), and $[x^{j r_1-s_2}] g_1(x)^{j}$ is a polynomial of degree at most $s_2$ in $j$. Also, $a_{s_3,p} (j-i)$ is a polynomial of degree $2 s_3$ in $j-i$ (so it is a polynomial of degree $2s_3$ in $j$). Thus this summand is a polynomial in $j$ of degree at most
$$s_1 + s_2 + 2s_3 \leq 2s_1 + 2s_2 + 2s_3 = p- 2i .$$
By (*), the only relevant products in the sum are for tuples where the inequality is an equality. This only occurs when $s_1 = s_2 = 0$ and $s_3 = \frac{p}2 - i$. Since the leading coefficient of $a_{\frac{p}2-i, p}(x)$ is $\frac{w^{p/2- i}}{(p/2-i)!} \neq 0$ by (3), we have
$$[\chi_n^{d_{p/2}(p)}] C_{i,j}(n,p) = [\chi_n^{d_i(p) - p/2 + i}] C_{i,j}(n,p) = \frac{ \delta^{p-i}  w^{\frac{p}2 - i} }{(p/2-i)!} j^{p-2i} + \eta_i(j)  $$
for some polynomial $\eta_i(x)$ of degree at most $p-2i-1$.\\

\noindent\underline{Step 3: Compute the coefficient of $\chi_n^{d_{p/2}(p)}$ in the entire numerator and get some cancellations}\\
We use our above results to compute the desired coefficient. Note that $1 + \frac1{2 \delta w} \neq 0$ by (1).
\begin{align*} & [\chi_n^{d_{p/2}(p)}] \sum \limits_{j=0}^ p {p \choose j} (-1)^{p-j} \sum \limits_{i=0}^{j} {j \brace j-i} C_{i,j} (n,p) \\
&\quad = \sum \limits_{j=0}^ p {p \choose j} (-1)^{p-j}  \sum \limits_{i=0}^{p/2} \bigg( \frac{j^{2i}}{i! 2^i} + \rho_i(j) \bigg) \bigg (  \frac{(\delta^2 w)^{\frac{p}2} (\delta w)^{-i}}{(p/2 -i)!} j^{p-2i} + \eta_i(j)) \bigg) \\
&\quad = (\delta^2 w)^{p/2} \sum \limits_{j=0}^ p {p \choose j} (-1)^{p-j}  \sum \limits_{i=0}^{p/2} \frac{1}{i! (p/2 - i)!} \bigg (\frac1{(2 \delta w)^i} j^p + \sigma_i(j)  \bigg) \\
&\quad\qquad \text{for some polynomial $\sigma_i(x)$ of degree at most $p-1$} \\
&\quad = (\delta^2 w)^{p/2}  \sum \limits_{i=0}^{p/2} \frac{1}{i! (p/2 - i)!} \sum \limits_{j=0}^ p {p \choose j} (-1)^{p-j}  \bigg (\frac1{(2 \delta w)^i} j^p + \sigma_i(j)  \bigg) \\
&\quad = (\delta^2 w)^{p/2}  \sum \limits_{i=0}^{p/2} \frac{1}{i! (p/2 - i)!} \bigg(\frac1{(2 \delta w)^i} p!   \bigg) \ \text{by Proposition 2.4.2} \\
& \quad= (\delta^2 w)^{p/2} \frac{p!}{(p/2)!} \sum \limits_{i=0}^{p/2} {p/2 \choose i} \frac1{(2 \delta w)^i}  \\
& \quad= (\delta^2 w)^{p/2} \frac{p!}{(p/2)!} \bigg ( 1 + \frac1{2 \delta w} \bigg)^{p/2}
\end{align*}

\noindent\underline{Step 4: Compute $\bE[(X_n - \bE[X_n])^p]/ \bE[(X_n - \bE[X_n])^2]^{p/2}$ and take the limit}\\
We will use big-O notation for simplicity, with $f(n) \in O(g(n))$ meaning $\lim \limits_{n \rightarrow \infty} \bigg| \frac{f(n)}{g(n)}\bigg | \in [0, \infty)$. Observe that any term in the numerator of $ \bE[(X_n - \bE[X_n])^p] $ where the degree of $\chi_n$ is higher than $d_{p/2} (p)$ will yield a polynomial in $j$ of degree strictly less than $p-2i$ as the coefficient. These will simplify to 0 when summed over $j$ (by Proposition 2.4.2). Since the $g_t$'s are monic polynomials we get 
$$ \bE[(X_n - \bE[X_n])^p] = \frac{ \frac{p!}{(p/2)!}  (\delta^2 w)^{p/2} \bigg ( 1 + \frac1{2 \delta w} \bigg)^{p/2}  \chi_n^{d_{p/2}(p)} + O(\chi_n^{d_{p/2}(p)-1})}{ \chi_n^{e(p)}  + O(\chi_n^{e(p)-1})}  .$$
In particular, we see that $$ \ \bE[(X_n - \bE[X_n])^2] = \frac{ 2  (\delta^2 w) \bigg ( 1 + \frac1{2 \delta w} \bigg)  \chi_n^{d_{1}(2)} + O(\chi_n^{d_{1}(2)-1})}{ \chi_n^{e(2)}  + O(\chi_n^{e(2)-1})}  .$$
It follows that \[ \frac{ \bE[(X_n - \bE[X_n])^p] }{ \bE[(X_n - \bE[X_n])^2]^{p/2}} = \frac{ \frac{p!}{(p/2)!}  (\delta^2 w)^{p/2} \bigg ( 1 + \frac1{2 \delta w} \bigg)^{p/2} \chi_n^{d_{p/2}(p) + \frac{p}2 e(2) } + O(\chi_n^{d_{p/2}(p) + (p/2) e(2) -1})}{ \bigg[2 \delta^2 w   \bigg ( 1 + \frac1{2 \delta w} \bigg) \bigg]^{p/2}  \chi_n^{e(p) + \frac{p}2 d_1(2)} + O( \chi_n^{e(p) + \frac{p}2 d_1(2)-1} )}  \]
where 
\begin{align*} & d_{p/2}(p) + (p/2) e(2) - e(p) - (p/2) d_1(2) \\
&\quad = (p - p/2) + e(p) + (p/2) e(2) - e(p) - (p/2) [(2-1) + e(2)] \\
&\quad = 0 .\end{align*}
Since $\lim \limits_{n \rightarrow \infty} \chi_n = \infty$ by (4), we see that $ \lim \limits_{n \rightarrow \infty} \frac{ \bE[(X_n - \bE[X_n])^p] }{ \bE[(X_n - \bE[X_n])^2]^{p/2}} =  \frac{p!}{(\frac{p}2)! 2^{p/2}}$ as desired. \\

\noindent \underline{Step 5: Consider the case where $p$ is odd}\\
Now suppose $p$ is odd. Note $d_{\lfloor p/2 \rfloor} (p) + (p/2) e(2) < d_{p/2}  (p) + (p/2) e(2) = e(p) + (p/2) d_1(2)$ (since $d_i(p)$ is a strictly increasing function of $i$). Applying the same argument as above thus gives a strictly lower degree for the numerator than the denominator. It then follows that $ \frac{ \bE[(X_n - \bE[X_n])^p] }{ \bE[(X_n - \bE[X_n])^2]^{p/2}} = O \left( \frac1{\sqrt{\chi_n}} \right) $, and hence $\lim \limits_{n \rightarrow \infty} \frac{ \bE[(X_n - \bE[X_n])^p] }{ \bE[(X_n - \bE[X_n])^2]^{p/2}} = 0$.\\
\end{proof} 

\noindent Let us discuss the conditions (1) - (4) in Theorem 3.1.1. First, (1) is a technical condition ensuring that the coefficient of $\chi_n^{d_{p/2}(p)}$ in the numerator of $\bE[(X_n - \bE[X_n])^p]$ is nonzero (to avoid computing coefficients of lower powers of $\chi_n$). Asymptotic normality may still hold when (1) is false; in fact other normality theorems do not impose such a condition. Nevertheless, (1) will hold most of the time and the effort required to check that (1) holds is offset by the ease of use of this theorem. Next, condition (2) is simply a condition on the overall degree of $\chi_n$ in $\bE[(X_n)_p]$, which is common in normality theorems.
 
The essential condition that makes this technique work is condition (3), which states that the coefficients of $f_j(x) \prod \limits_{1 \leq h \leq p, h \neq j} g_h(x)$ are polynomials in $j$ with specific degrees and leading coefficients. These particular values yield both the cancellations of high degree terms and the appropriate coefficients for asymptotically normal moments. However, it may be unwieldy to calculate the coefficient of a product of many polynomials; it turns out that there is a simpler sufficient condition on the coefficients of $f_p(x), g_p(x)$ that will imply condition (3). Finally, condition (4) allows us to cover a wider range of cases with this theorem; however, in our examples we will only work with $\chi_n = n$.

\subsection{Simplifying (3)}
The following is Theorem 3.1.1 with new conditions (1), (3) and with $w \neq 0$ replaced by $w_1- w_2 \neq 0$ (where $w$ came from condition (3) and now $w_1, w_2$ come from condition (3)).
\begin{theorem}[3.2.1]
Let $w_1, w_2 \in \bC, \delta \in \bC \setminus \{0\}$ be constants with $w_1 \neq w_2$ (1) and with \\ 
$2 \delta (w_1-w_2) \neq -1$ (1). Let $(q_p)_{p \geq 0}, (r_p)_{p \geq 0}$ be nonnegative integers satisfying 
$$q_0 = r_0 = 0 \ \mbox{and} \ q_p - r_p = p \ \forall p \in \bN \ (2)$$
and for every $p \geq 0$ let $f_p(x), g_p(x)$ be monic complex polynomials of degrees $q_p, r_p$ respectively.

For every $m \geq 0, p \geq 0$ define $\zeta_m(p) := [x^{q_p - m}] f_p(x)$ and $\tau_m(p):= [x^{r_p-m}] g_p(x)$ (the $m\textsuperscript{th}$ highest coefficients of the $f_p(x), g_p(x)$). Suppose $\zeta_m(p), \tau_m(p)$ are polynomials in $p$ of degree $2m$ with leading coefficient $\frac{w_1^m}{m!}, \frac{w_2^m}{m!}$ respectively $\forall m, p \geq 0$ (3). 

Let $(\chi_n)_{n \in \bN}$ be complex numbers such that $\lim \limits_{n \rightarrow \infty} \chi_n = \infty$ (4). If $(X_n)_{n \in \bN}$ are real random variables satisfying
$$\bE [(X_n)_p] =  \frac{ \delta^p f_p(\chi_n)}{g_p(\chi_n)}  \ \forall p \in \bN, n \in \bN, n > p \ (5)$$
$$\mbox{then} \ \lim_{n \rightarrow \infty} \frac{\bE [(X_n - \bE [X_n])^p]} {\bE [(X_n - \bE [X_n])^2]^{p/2}}= 
\begin{cases}
\frac{p!}{(p/2)! 2^{p/2}} & \mbox{$p$ is even} \\
0 & \mbox{$p$ is odd} \\
\end{cases} \ \mbox{so} \ \frac{X_n - \bE[X_n]}{\sqrt{Var(X_n)}}  \xrightarrow{d} N(0,1).$$
(Note: One of $w_1,w_2$ may be 0 (but not both), making the corresponding polynomial in $j$ of degree at most $2m-1$. This does not affect the proof.)
\end{theorem}

\begin{proof} We clearly seek to apply Theorem 3.1.1 with $w = w_1 - w_2$ satisfying $2 \delta w \neq -1$. It thus suffices to check that condition (3) from Theorem 3.1.1 holds, where for every $m \geq 0, p \geq 1, 0 \leq j \leq p$ we let $a_{m,p}(j):= [x^{q_j - r_j + \sum \limits_{h=1}^p r_h-m} ]f_j(x) \prod \limits_{1 \leq h \leq p, h \neq j} g_h(x)$. First, we must deal with the product of the $g_h(x)$'s.\\

\noindent  Fix any $p \geq 1$. For every $e \geq 0$ and $0 \leq j \leq p$, let $\widetilde{\tau}_{p,e}(j) := [x^{\sum \limits_{h=1}^p r_p - r_j -e}] \prod \limits_{1 \leq h \leq p, h \neq j} g_p(x)$. Note that the coefficients of the polynomial $\prod \limits_{h=1}^p g_h(x)$ are constants.  By a straightforward induction argument based on extracting coefficients from the equation 
$$  \prod \limits_{h=1}^p g_h(x) = g_j(x) \prod \limits_{1 \leq h \leq p, h \neq j}^p g_h(x) $$
and by using (3) we can prove the following claim.\\
\textbf{Claim:}  $\widetilde{\tau}_{p,e}(j)$ is a polynomial in $j$  of degree $2e$ with leading coefficient $\frac{(- w_2)^e}{e!}$ for every fixed $e \geq 0$\\

For this same fixed $p$ we will compute $a_{m,p}(j)$ as polynomials in $j$. For any $m \geq 0$,
\begin{align*}
a_{m,p}(j) &= \sum_{s + t = m; s ,t \geq 0} [x^{q_j - s}] f_j(x) \cdot  [x^{\sum \limits_{h=1}^p r_h - r_j - t}] \prod \limits_{1 \leq h \leq p, h \neq j} g_h(x)\\
&= \sum_{s + t = m; s,t \geq 0} \zeta_{s}(j) \widetilde{\tau}_{p,t}(j)\\
&= \sum_{s + t = m; s,t \geq 0} \bigg[\frac{w_1^{s}}{s!} j^{2s} + \rho_s(j) \bigg] \bigg[\frac{(-w_2)^{t}}{t!} j^{2t} + \sigma_t(j) \bigg]
\end{align*}
for some polynomials $\rho_s(x), \sigma_t(x)$ whose respective degrees are at most $2s-1, 2t-1$
\begin{align*}
&= \frac{j^{2m}}{m!}  \bigg[ \sum_{s=0}^m {m \choose s} w_1^s (-w_2)^{m-s} \bigg] + \sigma(j) \qquad \qquad
\end{align*}
for some polynomial $\sigma(x)$ of degree at most $2m-1$
$$ = \frac{(w_1 -w_2)^m}{m!} j^{2m}+ \sigma(j) \qquad \qquad \qquad \qquad \quad$$
$$= \frac{w^m}{m!} j^{2m} + \sigma(j) \qquad \qquad \qquad \qquad \qquad \qquad$$
where $w = w_1 - w_2 \neq 0$. Thus condition (3) holds as well, and the result follows by Theorem 3.1.1.\\
\end{proof}

\noindent To summarize, it is sufficient to check that the coefficients of $f_p(x), g_p(x)$ (namely $\zeta_m(p), \tau_m(p)$) are polynomials of degree $2m$ with leading coefficients of the form $\frac{w_1^m}{m!}, \frac{w_2^m}{m!}$ respectively. This is much simpler than computing coefficients of $f_j(n) \prod \limits_{1 \leq h \leq p, h \neq j} g_h(n)$, and we will exclusively use Theorem 3.2.1 in the specific examples highlighted in section 4.

\subsection{The General Result Applied to Products of Falling Factorials}
In this section, we apply Theorem 3.2.1 to random variables whose factorial moments are products and quotients of falling factorials in $n$ and other polynomials in $n$ of ``sufficiently small degree" (this is made precise below). This result is useful because countless real random variables have factorial moments of this form.

\begin{corollary}[3.3.1]
Let $k, l \in \bZ^{\geq 0}, k + l \geq 1$, $\gamma_1,\ldots, \gamma_k, \gamma_1',\ldots, \gamma_l' \in \bN$, $\beta_1,\ldots, \beta_k,\beta_1',\ldots,\beta_l' \in \bC$, $\alpha_1,\ldots, \alpha_k, \alpha_1',\ldots,\alpha_l', \sigma \in \bC \setminus \{0\}$ with 
$$\sum \limits_{s = 1}^k \frac{\gamma_s^2}{\alpha_s} \neq \sum \limits_{t=1}^l \frac{\gamma_t'^2}{\alpha_t'} \ (1) \ \mbox{and} \ \prod \limits_{t=1}^l \alpha_t'^{\gamma_t'}   \neq  \sigma \prod \limits_{s=1}^k \alpha_s^{\gamma_s} \bigg( \sum \limits_{s=1}^k \frac{\gamma_s^2}{\alpha_s}  -  \sum \limits_{t=1}^l \frac{\gamma_t'^2}{\alpha_t'}  \bigg)  \ (2).$$
Let $(u_p)_{p \geq 0}, (v_p)_{p \geq 0}$ be nonnegative integers satisfying 
$$u_0 = v_0 = 0,  u_p - v_p = p (u_1 - v_1) \ \forall p \in \bN \ (3) \ \mbox{and} \ u_1 - v_1 +  \sum \limits_{s = 1}^k \gamma_s   -  \sum \limits_{t = 1}^l \gamma_t'= 1 \ (4)$$
and let $y_p(x), z_p(x)$ be monic complex polynomials of degrees $u_p, v_p$ respectively.

For every $m \geq 0, p \geq 0$ define $b_m(p):= [x^{u_p - m}] y_p(x)$ and $c_m(p):= [x^{v_p - m}] z_p(x)$ (the $m$\textsuperscript{th} highest coefficients of $y_p, z_p$, respectively). Suppose $b_m(x), c_m(x)$ are complex polynomials of degree at most $2m-1$ for all $m \geq 1$.

If $(X_n)_{n \in \bN}$ are real random variables satisfying
$$\bE [(X_n)_p] =  \frac{ \sigma^p y_p(n) \prod \limits_{s=1}^k (\alpha_s n + \beta_s)_{\gamma_s p }} {z_p(n) \prod \limits_{t=1}^l (\alpha_t' n + \beta_t')_{\gamma_t' p }}  \ \forall p \in \bN, n \in \bN, n > p \ (5)$$
then $\frac{X_n - \bE[X_n]}{\sqrt{Var(X_n)}}  \xrightarrow{d} N(0,1)$.
\end{corollary}

\begin{proof}
\underline{Step 1: Set-up the conditions of Theorem 3.2.1}\\
$$\mbox{Let} \ \chi_n = n \ \forall n \in \bN, \delta = \frac{\sigma \prod \limits_{s=1}^k \alpha_s^{\gamma_s}} {\prod \limits_{t=1}^l \alpha_t'^{\gamma_t'}}, f_p(x) = \frac{y_p(x) \prod \limits_{s=1}^k (\alpha_s x + \beta_s)_{\gamma_s p } }{ \bigg( \prod \limits_{s=1}^k \alpha_s^{\gamma_s} \bigg)^p},  g_p(x) = \frac{ z_p(x) \prod \limits_{t=1}^l (\alpha_t' x + \beta_t')_{\gamma_t' p }}{ \bigg( \prod \limits_{t=1}^l \alpha_t'^{\gamma_t'}\bigg)^p}$$
and $q_p = u_p + p \sum \limits_{s=1}^k \gamma_s, r_p = v_p + p \sum \limits_{t=1}^l \gamma_t'$ for all $p \geq 0$. We clearly have $q_0 = r_0 = 0$. Also, note that  $f_p(x), g_p(x)$ are monic polynomials of degrees $q_p, r_p$ respectively, that $\bE [(X_n)_p] =  \frac{ \delta^p f_p(n)}{g_p(n)}$ for all $p \geq 0, n \in \bN, n > p $, and that
$$q_p- r_p = u_p - v_p + p \sum \limits_{s=1}^k \gamma_s - p \sum \limits_{t=1}^l \gamma_t' = p \bigg(u_1 - v_1 + \sum \limits_{s=1}^k \gamma_s - \sum \limits_{t=1}^l \gamma_t' \bigg) = p \ \mbox{by (3) and (4)}.$$
For every $m,j \geq  0$ define $\zeta_m(j):= [x^{q_j- m}] f_j(x), \tau_m(j) := [x^{r_j - m}] g_j(x)$. \\
\underline{Step 2: Compute $\zeta_{m}(j)$ as a polynomial in $j$ for each fixed $m\geq 0$}\\
$$\mbox{We have} \ \zeta_m(j) = \sum_{\substack{m_0, \ldots, m_k \geq 0\\ m_0 + \ldots + m_k = m}} [x^{u_j - m_0}] y_j(x) \cdot \prod_{s=1}^k [x^{\gamma_s j - m_s}] \frac1{\alpha_s^{\gamma_s j}} (\alpha_s x + \beta_s)_{\gamma_s j } $$
Fix any such $(m_0,m_1,\ldots, m_k)$ and determine an expression for each of the coefficients in the corresponding summand. First, $[x^{u_j - m_0}] y_j(x) = b_{m_0}(j)$ which is a polynomial in $j$ of degree at most $\max(2m_0-1,0)$. On the other hand, for every $1 \leq s \leq k$ and $j \geq 0$,  
\begin{align*}
(\alpha_s x + \beta_s)_{\gamma_s j } &= \sum_{e=0}^{\gamma_s j} \stirling{\gamma_s j}{\gamma_s j - e} (-1)^{e} (\alpha_s x + \beta_s)^{\gamma_s j - e} \\
\Rightarrow  [x^{\gamma_s j - m_s}] \frac1{\alpha_s^{\gamma_s j}} (\alpha_s x + \beta_s)_{\gamma_s j } &= \frac1{\alpha_s^{\gamma_s j}} \sum_{e=0}^{m_s} \stirling{\gamma_s j}{\gamma_s j - e} (-1)^{e} [x^{\gamma_s j - m_s}] (\alpha_s x + \beta_s)^{\gamma_s j - e} 
\end{align*}
since $\stirling{\gamma_s j}{\gamma_s j - e} =0$ whenever $e > \gamma_s j$ and $[x^{\gamma_s j - m_s}] (\alpha_s x + \beta_s)^{\gamma_s j - e} =0$ whenever $e > m_s$
\begin{align*}
& \qquad  \qquad \qquad \qquad \qquad \qquad \qquad = \frac1{\alpha_s^{\gamma_s j}} \sum_{e=0}^{m_s} \stirling{\gamma_s j}{\gamma_s j - e} (-1)^{e} \alpha_s^{\gamma_sj -m_s} \beta_s^{m_s-e} {\gamma_s j - e \choose \gamma_s j -m_s}  \\
&\qquad  \qquad \qquad \qquad \qquad \qquad \qquad =  \sum_{e=0}^{m_s} \stirling{\gamma_s j}{\gamma_s j - e} (-1)^{e} \alpha_s^{ -m_s} \beta_s^{m_s-e} {\gamma_s j - e \choose \gamma_s j -m_s}  
\end{align*}
where each $\stirling{\gamma_s j}{\gamma_s j - e}$ is a polynomial in $\gamma_s j$ of degree $2e$ with leading coefficient $\frac1{2^e e!}$ by Proposition 2.3.5, i.e. it is a polynomial in $j$ of degree $2e$ with leading coefficient $\frac{\gamma_s^{2e}}{2^e e!}$. Moreover, by Remark 2.1.2, ${\gamma_s j - e \choose \gamma_s j -m_s}$ is a polynomial in $\gamma_s j - e$ of degree $m_s-e$ (i.e. it is a polynomial in $j$ of degree $m_s-e$). Thus each $\stirling{\gamma_s j}{\gamma_s j - e} (-1)^{e} \alpha_s^{-m_s} \beta_s^{m_s-e} {\gamma_s j - e \choose \gamma_s j -m_s} $ is a polynomial in $j$ of degree $m_s + e$. It follows that 
$[x^{\gamma_s j - m_s}]  \frac1{\alpha_s^{\gamma_s j}} (\alpha_s x + \beta_s)_{\gamma_s j }$ is a polynomial in $j$ of degree $2m_s$ (corresponding to the $e=m_s$ term of the sum) with leading coefficient $ (-1)^{m_s} \alpha_s^{-m_s} \frac{\gamma_s^{2m_s}}{2^{m_s} m_s!}$. Thus each summand is a polynomial in $j$ of degree  
$$\max(2m_0-1,0) + 2 \sum \limits_{s=1}^k m_s \leq 2 \sum \limits_{s=0}^k m_s = 2m,$$
and $\zeta_m(j)$ is a polynomial in $j$ of degree $2m$ with leading coefficient given by the sum over the tuples with $m_0 =0$, namely\\
\begin{align*} \sum_{\substack{m_1, \ldots, m_k \geq 0\\ m_1 + \ldots + m_k = m}} \prod_{s=1}^k \bigg[\alpha_s^{-m_s} (-1)^{m_s}  \frac{\gamma_s^{2m_s}}{2^{m_s} m_s!} \bigg] &=  \frac{(-1)^m}{2^m m!} \sum_{\substack{m_1, \ldots, m_k \geq 0\\ m_1 + \ldots + m_k = m}} {m \choose m_1, \ldots,m_k } \prod_{s=1}^k \bigg(\frac{\gamma_s^2}{\alpha_s} \bigg)^{m_s} \\
\quad &= \frac1{m!} \bigg(- \frac12 \sum \limits_{s=1}^k \frac{\gamma_s^2}{\alpha_s} \bigg)^m
\end{align*}
Note: it is possible to have $\sum \limits_{s=1}^k \frac{\gamma_s^2}{\alpha_s} =0$; this just means that $\zeta_m(j)$ is a polynomial in $j$ of degree at most $2m-1$.\\
\underline{Step 3: Compute $\tau_{m}(j)$ as a polynomial in $j$ for each fixed $m \geq 0$}\\
By the exact same argument as in Step 2, we get that $\tau_{m}(j)$ is a polynomial in $j$ of degree $2m$ with leading coefficient $\frac1{m!} \bigg(- \frac12 \sum \limits_{t=1}^l \frac{\gamma_t'^2}{\alpha_t'} \bigg)^m$. As before, it is possible that $\sum \limits_{t=1}^l \frac{\gamma_t'^2}{\alpha_t'} = 0$.\\
\underline{Step 4: Apply Theorem 3.2.1}\\
Let $w_1 = - \frac12 \sum \limits_{s=1}^k \frac{\gamma_s^2}{\alpha_s}$ and $w_2 = - \frac12 \sum \limits_{t=1}^l \frac{\gamma_t'^2}{\alpha_t'}$. By (1) we have $w_1-w_2 = \frac12 \bigg( - \sum \limits_{s=1}^k \frac{\gamma_s^2}{\alpha_s} + \sum \limits_{t=1}^l \frac{\gamma_t^2}{\alpha_t} \bigg) \neq 0$ and by (2) we have
$$2 \delta (w_1-w_2) = \frac{\sigma \prod \limits_{s=1}^k \alpha_s^{\gamma_s}} {\prod \limits_{t=1}^l \alpha_t'^{\gamma_t'}}  \bigg( - \sum \limits_{s=1}^k \frac{\gamma_s^2}{\alpha_s} + \sum \limits_{t=1}^l \frac{\gamma_t^2}{\alpha_t} \bigg) \neq -1.$$
All conditions of Theorem 3.2.1 are thus satisfied, and the conclusion follows.\\
\end{proof}

\noindent Note that Corollary 3.3.1 still holds with the $n$ replaced by $\chi_n$ (as in Theorem 3.2.1), but we will not need this in the next section. Moreover, in practice it turns out that many discrete random variables have factorial moments that are rational expressions with some falling factorials, which fit the situation described in Corollary 3.3.1. 

\begin{center}
\section{Applications}

\end{center}

\subsection{Basic Classical Applications}
In \cite{cumulants}, Bobecka et al. give the formulas for the factorial moments of some classical discrete distributions.
\begin{itemize}
\item Binomial Distribution: Fix $q \in (0,1)$. If $X_n \sim B(n,q)$ then $\bE[(X_n)_p] = (n)_p q^p \ \forall p \geq 0$
\item Negative Binomial Distribution:  Fix $q \in (0,1)$. If $X_n \sim NB(n,q)$ then \\ $\bE[(X_n)_p] = (-n)_p \bigg(\frac{q-1}q \bigg)^p$  $\forall p \geq 0$
\item Allocation of balls into boxes: Fix $\lambda, r \in \bR$, and sample uniformly from the distinct ways of allocating $n$ indistinguishable balls into $\lambda n$ distinguishable boxes. If $X_n$ is the number of boxes with exactly $r$ balls, then $\bE[(X_n)_p] = \frac{(\lambda n)_p (\lambda n - 1)_p (n)_{pr}}{(n+\lambda n - 1)_{p(r+1)}}$  $\forall p \geq 0$.
\end{itemize}
It is straightforward to check that all of these factorial moments satisfy the conditions in Corollary 3.3.1 and so we get $\frac{X_n - \bE[X_n]}{\sqrt{Var \ X_n}} \xrightarrow{d} N(0,1).$

\subsection{Lagrange's Implicit Function Theorem and Factorial Moments}
The following two theorems, discussed by Flajolet and Sedgewick in \cite{ancomb}, will simplify the computation of factorial moments in the examples we will discuss next.

\begin{theorem}[4.2.1] (Lagrange's Implicit Function Theorem)\\
Suppose $g(u) = u \phi(g(u))$ where $\phi(g)$ is a formal power series in $g$. Then  
$$[u^n] g(u) = \frac1n [x^{n-1}]  \phi(x)^n \ \forall n \in \bN.$$
\end{theorem}

\begin{definition}[4.2.2] The  \underline{probability generating function} of an integer-valued random variable $X$ is $G(w) := \bE[w^X]= \sum \limits_{i= - \infty}^{\infty} \mbox{Pr}(i) w^i$. 
\end{definition}

\begin{theorem}[4.2.3] If $X$ is an integer-valued random variable with probability generating function $G(w)$ then the $p\textsuperscript{th}$ factorial moment of $X$ is given by $\bE[(X)_p] = \frac{d^p}{dw^p} G(w) \bigg|_{w=1}$.
\end{theorem}

\subsection{$k$-Divisible Non-Crossing Partitions}
As an example of random variables with factorial moments that satisfy the conditions of Corollary 3.3.1, we will look at statistics on $k$-divisible non-crossing partitions, discussed by Arizmendi in \cite{arizmendi}.

\begin{definition}[4.3.1] For any $n \in \bN$, a \underline{partition} of $\{1,\ldots, n\}$ is a set $\pi = \{V_1, \ldots, V_r\}$ where $\emptyset \neq V_j \subseteq \{1, \ldots, n\}$ are pairwise disjoint and $\bigcup \limits_{j=1}^{r} V_j = \{1,\ldots, n\}$. $V_j$ are the \underline{blocks} of $\pi$. \\
A partition $\pi$ of $\{1,\ldots, n\}$ is called \underline{non-crossing} if there do not exist $1 \leq a < b < c <d \leq n$ with $a,c$ in the same block $V$ of $\pi$ and with $b,d$ in the same block $W$ of $\pi$ where $V \neq W$. \\
For any $k \geq 1$, a partition $\pi$ is \underline{$k$-divisible} if the size of each block of $\pi$ is divisible by $k$. The set of $k$-divisible non-crossing partitions of $\{1,\ldots,kn\}$ is denoted by $NC^k(n)$.  
\end{definition}

For example, the following are a (1-divisible) non-crossing partition of $\{1,\ldots, 8\}$ and a 3-divisible non-crossing partition of $\{1,\ldots, 6\}$:
\begin{center}
\begin{tikzpicture}[thick,font=\small]
     \path (0,0) node (a) {1}
           (1,0) node (b) {2}
           (2,0) node (c) {3}
           (3,0) node (d) {4}
           (4,0) node (e) {5}
           (5,0) node (f) {6}
           (6,0) node (g) {7}
           (7,0) node (h) {8};
     \draw (a) -- +(0,0.75) -| (d);
     \draw (b) -- +(0,0.5) -| (c);
     \draw (f) -- +(0,0.5) ;
     \draw (e) -- +(0,0.75) -| (g);
     \draw (h) -- +(0,0.75) ;
   \end{tikzpicture}
\end{center}
\begin{center}
\begin{tikzpicture}[thick,font=\small]
     \path (0,0) node (a) {1}
           (1,0) node (b) {2}
           (2,0) node (c) {3}
           (3,0) node (d) {4}
           (4,0) node (e) {5}
           (5,0) node (f) {6};
     \draw (a) -- +(0,0.75) -| (b);
     \draw (b) -- +(0,0.75) -| (f);
     \draw (c) -- +(0,0.5) -| (d);
     \draw (d) -- +(0,0.5) -| (e);
   \end{tikzpicture}
\end{center}

To use the theorems in section 4.2, we first need a generating function $g$ for non-crossing partitions that satisfies $g(u) = u \phi(g(u))$.

\begin{proposition}[4.3.2] $f(u,w_1,w_2,\ldots) = \sum \limits_{j=1}^{\infty} u^j w_j (1+f(u,w_1,w_2,\ldots) )^j$ is a generating function for all non-crossing partitions (of all sizes), where the weight of $u$ is the size and the weight of $w_j$ is the number of blocks of size $j$ (for all $j \geq 1$).

\end{proposition}

\begin{proof}
Consider any non-crossing partition $\pi$ (of some size), let $V_1 \in \pi$ be the block containing 1 and let $j := |V_1|$ so $j \geq 1$. Observe that because of the non-crossing condition, $V_1$ splits $\pi$ into $j$ disjoint (possibly empty) ``pockets", each of which is either empty or contains a non-crossing partition of some (positive) size. For example, a non-crossing partition (of some arbitrary size) with $V_1 = \{1,4,5,7\}$ can be decomposed as shown below.

\begin{center}
\begin{tikzpicture}[thick,font=\small]
     \path (0,0) node (a) {1}
           (1,0) node (b) {2}
           (2,0) node (c) {3}
           (3,0) node (d) {4}
           (4,0) node (e) {5}
	(1.5, -1) node {pocket}
	(1.5, -1.5) node {of size 2}
	(3.5, -1) node {pocket}
	(3.5, -1.5) node {of size 0}
	(5.2, -1) node {pocket}
	(5.2, -1.5) node {of size 1}
	(8, -1) node {pocket}
	(8, -1.5) node {of some size}
           (5,0) node (f) {6}
           (6,0) node (g) {7}
           (9,0) node (inv) {}
           (7,0) node (h) {8}
           (8,0) node (rest) {...};
     \draw (a) -- +(0,0.75) -| (d);
     \draw (e) -- +(0,0.75) -| (g);
     \draw (d) -- +(0,0.75) -| (e);
\draw [ decoration={
        brace,
        mirror,
        raise=0.4cm
    },  decorate] (a) -- (d) ;
\draw [ decoration={
        brace,
        mirror,
        raise=0.4cm
    },  decorate] (d) -- (e) ;
\draw [ decoration={
        brace,
        mirror,
        raise=0.4cm
    },  decorate] (e) -- (g) ;
\draw [ decoration={
        brace,
        mirror,
        raise=0.4cm
    },  decorate] (h) -- (inv) ;

   \end{tikzpicture}
\end{center}

Therefore a generating function for all non-crossing partitions (of all sizes) with $u$ recording the size and $w_j$ counting the blocks of size $j$ (for all $j \geq 1$) is
$$ f(u,w_1,w_2,\ldots) = \sum \limits_{j=1}^{\infty} u^j w_j (1+f(u,w_1,w_2,\ldots) )^j$$
where the $(1+f)^j$ represent these $j$ ``pockets" (which are either empty or a copy of $f$), where $u^j$ comes from the $j$ elements in $V_1$ and where $w_j$ comes from the block $V_1$ of size $j$.\\
\end{proof}

It follows that for a fixed $k \in \bN$, the generating function for all $k$-divisible non-crossing partitions is 
$$f_k(u,w_1,w_2,\ldots) = \sum \limits_{j=1}^{\infty} u^{jk} w_{jk} (1+f_k(u,w_1,w_2,\ldots) )^{jk}.$$
We want to apply Lagrange's Implicit Function Theorem on $g_k(u,w_1,\ldots) := u(1+f_k(u,w_1,\ldots))$ and $\phi(x) := 1 +  \sum \limits_{j=1}^{\infty} w_{jk} x^{jk} $ so that
$$g_k(u,w_1,\ldots) = u(1+f_k(u,w_1,\ldots)) = u \bigg(1 +  \sum \limits_{j=1}^{\infty}w_{jk} g_k(u,w_1,w_2,\ldots)^{jk} \bigg) = u \phi(g_k(u,w_1,\ldots)).$$

\begin{proposition}[4.3.3]
$|NC^k(n)| = \frac1{kn+1} {(k+1)n \choose n} = C_n^{(k)} \ \forall n, k \in \bN$, which are known as the Fuss-Catalan numbers.
\end{proposition}
\begin{proof}
Fix $k \in \bN$. Note that $|NC^k(n)| $ is just the coefficient of $u^{kn}$ in $f_k(u,1,\ldots)$. Set all $w_{jk}$ to 1. To simplify notation, let $f_k(u) := f_k(u,1,1,\ldots$), let $g_k(u) := g_k(u,1,1,\ldots)$ and let $\phi(x) := 1 + \sum \limits_{j=1}^{\infty} x^{jk} = (1-x^k)^{-1}$ so that $g_k(u) = u ( 1+ f_k(u)) = u \phi(g_k(u))$. Then
\begin{align*}
|NC^k(n)| &= [u^{kn}] f_k(u)\\
&= [u^{kn+1}] g_k(u)\\
&= \frac1{kn+1} [x^{kn}] \phi(x)^{kn+1} \ \mbox{by Theorem 4.2.1}\\
&= \frac1{kn+1} [x^{kn}] (1-x^k)^{-(kn+1)}\\
&= \frac1{kn+1} [y^{n}] (1-y)^{-(kn+1)}
\end{align*}
\begin{align*}
\qquad \qquad \qquad \qquad \qquad &= \frac1{kn+1} {-(kn+1) \choose n} (-1)^{n} \ \mbox{by the Binomial Theorem}\\
&= \frac1{kn+1} {(kn+1) + n - 1 \choose n}\\
&=\frac1{kn+1} {(k+1)n \choose n}
\end{align*}
\end{proof}

\noindent First, for any $n, k \in \bN$, consider the random variable $X_n^{(k)}(\pi)$ counting the number of blocks of $\pi$ (with $\pi \in NC^k(n)$ chosen uniformly at random). We will compute the factorial moments and then apply Corollary 3.3.1.

\begin{proposition}[4.3.4]
$$\bE[(X_n^{(k)})_p] = \frac{(kn+1)_p  (n)_{p}} {(k n+ n)_{ p}} \ \forall n,k,p \in \bN \ \mbox{with} \ p \leq nk$$
\end{proposition}
\begin{proof}
Fix any such $n, k,p$. Set $w_{jk} = w \ \forall j \geq 1$. For a fixed $w$, to simplify notation, let $f_k(u) := f_k(u,w,w,\ldots)$, let $g_k(u) := g_k(u,w,w,\ldots)$ and let $\phi(x) = 1 + \sum \limits_{j=1}^{\infty} w x^{jk} = 1 + \frac{wx^k}{1-x^k}$ so that $g_k(u) = u(1+f_k(u)) = u \phi(g_k(u))$. The probability generating function of $X_n^{(k)}$ is 
$G(w) = \frac1{C_n^{(k)}} [u^{kn}] f_k(u) =\frac1{C_n^{(k)}} [u^{kn+1}] g_k(u) = \frac1{C_n^{(k)}} [x^{kn}] \frac1{kn+1} \phi(x)^{kn+1}$ (by Theorem 4.2.1), so
\begin{align*}
\bE[(X_n^{(k)})_p] &=  \frac{d^p}{dw^p} \bigg( \frac1{C_n^{(k)}}  [x^{kn}] \frac1{kn+1} \phi(x)^{kn+1} \bigg) \bigg|_{w=1} \ \mbox{by Theorem 4.2.3} \\
&= \frac1{C_n^{(k)} (kn+1)}  [x^{kn}] \frac{d^p}{dw^p} \bigg( 1 + \frac{wx^k}{1-x^k} \bigg)^{kn+1} \bigg |_{w=1}\\
&= \frac1{C_n^{(k)}(kn+1)}  [x^{kn}] \bigg(1 + \frac{x^k}{1-x^k}  \bigg)^{kn+1 - p} (kn+1)_p \bigg( \frac{x^k}{1-x^k}  \bigg)^p \\
&= \frac{(kn+1)_p}{C_n^{(k)}(kn+1)}  [x^{kn- pk}] (1-x^k)^{-(kn+1 )}  \\
&= \frac{(kn+1)_p}{C_n^{(k)}(kn+1)}  [y^{n- p}] (1-y)^{-(kn+1)}  \\
&= \frac{(kn+1)_p}{C_n^{(k)}(kn+1)} {-(kn+1) \choose n- p} (-1)^{n-p} \ \mbox{by the Binomial Theorem}\\
&= \frac{(kn+1)_p}{C_n^{(k)}(kn+1)} {(kn+1) + (n-p) - 1 \choose n- p}\\ 
&= \frac{(kn+1)_p  (n)_{p}} {(k n+ n)_{ p}}
\end{align*}
\end{proof}
\noindent We now apply Corollary 3.3.1 to get the desired asymptotic normality.

\begin{theorem}[4.3.5]
Fix any $k \in \bN$. Then $\frac{X_n^{(k)} - \bE[X_n^{(k)}]}{\sqrt{Var \ X_n^{(k)}}} \xrightarrow{d} N(0,1).$
\end{theorem}
\begin{proof}
We want to apply Corollary 3.3.1 with $u_p = v_p = 0 = \deg y_p = \deg z_p$ and with \ \ $\sigma = y_p(x) = z_p (x) = 1$. Note that (3) and (4) hold trivially. Next, we check condition (1) by proving $0 < $ LHS $-$ RHS of (1). We have
$$g_{k} := \mbox{LHS $-$ RHS of (1)} = \frac{1^2}k + \frac{1^2}1 - \frac{1^2}{k+1} = \frac1{k(k+1)} + 1 > 0 $$
 Thus (1) holds. To check condition (2), we check that $0 < $ LHS $-$ RHS of (2). We have
$$h_{k} := \mbox{LHS $-$ RHS of (2)} = (k+1)^{1} - k^1 \cdot 1^1 \cdot g_{k} = (k+1) - k \bigg(\frac1{k(k+1)} + 1 \bigg) = 1 - \frac1{k+1} > 0$$
Therefore (2) holds as well. The desired conclusion then follows by Corollary 3.3.1. \\
\end{proof}

\noindent Next, we consider a less trivial class of random variables on $NC^k(n)$, namely block size statistics. For any $n, k, t \in \bN$ with $ t \leq n$, consider the random variable $Y_n^{(t,k)}$ which counts the number of blocks of size $tk$ in a $\pi \in NC^k(n)$ chosen uniformly at random. This is the random variable to which we will apply Corollary 3.3.1. To do this, we first compute its factorial moments.
\begin{proposition}[4.3.6]
$$\bE[(Y_n^{(t,k)})_p] = \frac{(kn+1)_p (kn)_p (n)_{pt}} {(k n+ n)_{pt + p}} \ \forall n,k,p,t \in \bN \ \mbox{with} \ t \leq n \ \mbox{and} \ p \leq nk.$$
\end{proposition}
\begin{proof}
Fix any such $n, k,p,t$. Set $w_{tk} = w, w_{jk} = 1 \ \forall j \neq t$. For a fixed $w$, to simplify notation, let $f_k(u) = f_k(u,1,\ldots,1,w,1,\ldots)$, let $g_k(u) = g_k(u,1,\ldots,1,w,1,\ldots)$ and let $\phi(x) = 1 + \sum \limits_{j=1}^{\infty} x^{jk} + (w-1) x^{tk} = \frac1{1-x^k} + (w-1)x^{tk}$ so that $g_k(u) = u(1+f_k(u)) = u \phi(g_k(u))$. Then the probability generating function of $Y_n^{(t,k)}$ is 
$$G(w) =\frac1{C_n^{(k)}} [u^{kn}] f_k(u) =\frac1{C_n^{(k)}} [u^{kn+1}] g_k(u) = \frac1{C_n^{(k)}} [x^{kn}] \frac1{kn+1} \phi(x)^{kn+1} \ \mbox{by Theorem 4.2.1, so}$$
\begin{align*}
\bE[(Y_n^{(t,k)})_p] &= \frac{d^p}{dw^p}\bigg( \frac1{C_n^{(k)}}  [x^{kn}] \frac1{kn+1} \phi(x)^{kn+1}\bigg)  \bigg|_{w=1} \ \mbox{by Theorem 4.2.3} \\
&= \frac1{C_n^{(k)}(kn+1)} [x^{kn}]  \frac{d^p}{dw^p}  \bigg( \frac1{1-x^k} + (w-1) x^{tk} \bigg)^{kn+1} \bigg |_{w=1}\\
&= \frac{(kn+1)_p}{C_n^{(k)}(kn+1)}  [x^{kn- ptk}] (1-x^k)^{-(kn+1 - p)}  \\
&= \frac{(kn+1)_p}{C_n^{(k)}(kn+1)}  [y^{n- pt}] (1-y)^{-(kn+1 - p)}  \\
\qquad \qquad \qquad \qquad &= \frac{(kn+1)_p}{C_n^{(k)}(kn+1)} {-(kn+1-p) \choose n- pt} (-1)^{n-pt} \ \mbox{by the Binomial Theorem}\\
&= \frac{(kn+1)_p}{C_n^{(k)}(kn+1)} {(kn+1-p) + (n-pt) - 1 \choose n- pt}\\ 
&= \frac{(kn+1)_p (kn)_p (n)_{pt}} {(k n+ n)_{pt + p}}
\end{align*}
\end{proof}

\noindent We now apply Corollary 3.3.1 to get the desired asymptotic normality. \\

\noindent \textbf{Theorem} (4.3.7). \ Fix any $t, k \in \bN$. Then $\frac{Y_n^{(t,k)} - \bE[Y_n^{(t,k)}]}{\sqrt{Var \ Y_n^{(t,k)}}} \xrightarrow{d} N(0,1).$
\begin{proof}
We want to apply Corollary 3.3.1 with $u_p = v_p = 0 = \deg y_p = \deg z_p$ and with \ \ $\sigma = y_p(x) = z_p (x) = 1$. Note that (3) holds trivially and $1 + 1 + t - (1 + t) = 1$ so (4) holds. Next, we check condition (1) by proving $0 < $ LHS $-$ RHS of (1). We have
$$g_{k,t} := \mbox{LHS $-$ RHS of (1)} = \frac{1^2}k + \frac{1^2}k + \frac{t^2}1 - \frac{(t+1)^2}{k+1} = \frac{2}{k} + t^2 - \frac{(t+1)^2}{k+1} $$
$$= \frac{2(k+1) + t^2 (k^2 + k) - k (t^2 +2t+1)}{k (k+1)} = \frac{t^2k^2 - 2tk +k + 2}{k(k+1)} = \frac{(tk - 1)^2 + k+1}{k(k+1)} > 0 $$
 Thus (1) holds. To check condition (2), we check that $0 < $ LHS $-$ RHS of (2). We have
$$h_{k,t} := \mbox{LHS $-$ RHS of (2)} = (k+1)^{t+1} - k^1 \cdot k^1 \cdot 1^t \cdot g_{k,t} = (k+1)^{t+1} - k^2 g_{k,t}$$
$$= \frac{(k+1)^{t+2} -k (tk-1)^2 - k^2 - k}{k+1} = \frac{(k+1)^{t+2} - t^2 k^3 + k^2(2t-1) - 2k}{k+1} $$
Observe that $f(x) := (x+2) (x+1) x - 6 x^2$ has roots $0,1, 2$ and is increasing for $x > 2$ so $f(x) > 0 \ \forall x > 2$ and thus ${ t+2 \choose 3} \geq t^2 \ \forall t \in \bN$. By inspecting the coefficients of $k,k^3$ in $(k+1)^{t+2}$ we get $h_{k,t} > 0$ for this arbitrary $k,t \geq 1$. Therefore (2) holds. \\
The desired conclusion follows by Corollary 3.3.1.\\
\end{proof}

\subsection{Rooted Plane Trees}
There is a natural bijection between rooted plane trees and noncrossing partitions arising from their generating functions, which will give us other statistics to which we can apply our result.

\begin{definition}[4.4.1] A \underline{tree} is a connected undirected (non-empty) graph with no cycles. A \underline{rooted tree} is a tree with a node labelled as the ``root" and edges directed away from the root. The \underline{children} of a node are its direct successors. A \underline{rooted plane tree} is a rooted tree with an ordering added to each node's children. The \underline{outdegree} of a node is the number of children. A \underline{leaf} is a vertex with outdegree 0. Let $T^k(n)$ be the number of rooted plane trees with $kn$ edges and all outdegrees of nodes being multiples of $k$. 
\end{definition}

\noindent For example, the following is a rooted plane tree with left-to-right ordering of children, with 11 edges and with the labels showing the outdegrees of each node:
\begin{center}
\begin{tikzpicture}
[level distance=10mm,
 my circle/.style={draw, fill, circle, minimum size=5pt, inner sep=0pt},
level 1/.style={sibling distance=30mm},
level 2/.style={sibling distance=20mm},
level 3/.style={sibling distance=10mm}]
\node[my circle]  {} [grow'=up]
child {node[my circle] {}
child {node[my circle] {}
child {node[my circle] {}}
child {node[my circle] {}}
}
child {node[my circle] {}
child {node[my circle] {}}
child {node[my circle] {}}
}
}
child {node[my circle] {}
child {node[my circle] {}
child[missing]
child {node[my circle] {}}
}
child {node[my circle] {}}};
\path (0.5,0) node (a) {2};
\path (1.9,1) node (a) {2};
\path (-1.9,1) node (a) {2};
\path (2.9,2) node (a) {0};
\path (1.1,2) node (a) {1};
\path (-0.2,2) node (a) {2};
\path (-3,2) node (a) {2};
\path (1.3,3) node (a) {0};
\path (0.3,3) node (a) {0};
\path (-0.7,3) node (a) {0};
\path (-1.7,3) node (a) {0};
\path (-2.7,3) node (a) {0};

\end{tikzpicture}
\end{center}

In the next proposition, we give a generating function for rooted plane trees (for example, as found in \cite{ancomb} by Flajolet and Sedgewick).

\begin{proposition}[4.4.2] $t(v,z_0,z_1,\ldots) = \sum \limits_{j=0}^{\infty} v^j z_j t(v,z_0,z_1,\ldots)^j$ is a generating function for all nonempty rooted plane trees, where the weight of $v$ is the number of edges and the weight of each $z_j$ is the number of vertices with outdegree $j$ (for all $j \geq 0$).

\end{proposition}

\begin{proof}
Consider any nonempty rooted plane tree $T$ (of some size), and let the root $R$ have outdegree $j$ for some $j \geq 0$. Observe that each of the $j$ subtrees with one of the children of $R$ as the root is just an arbitrary nonempty rooted plane tree. For example, a rooted plane tree with a root of outdegree 4 looks like:
\begin{center}
\begin{tikzpicture}
[level distance=10mm,
 my circle/.style={draw, fill, circle, minimum size=5pt, inner sep=0pt},
itria/.style={  draw,shape border uses incircle,  isosceles triangle,shape border rotate=90},
level 1/.style={sibling distance=10mm},
level 2/.style={sibling distance=20mm},
level 3/.style={sibling distance=10mm}]
\node[my circle]  {} [grow'=up]
child {node[itria] {}}
child {node[itria] {}}
child {node[itria] {}}
child {node[itria] {}};
\path (0.5,0) node (a) {R};

\end{tikzpicture}
\end{center}
where the triangles represent arbitrary nonempty rooted plane trees. Therefore a generating function for all nonempty rooted plane trees with $v$ counting the edges and $z_j$ counting the vertices with outdegree $j$ (for all $j \geq 0$) is 
$$t(v,z_0,z_1,\ldots) = \sum \limits_{j=0}^{\infty} v^j z_j t(v,z_0,z_1,\ldots)^j$$
where $v^j$ comes from the $j$ edges going out from $R$ and $z_j$ comes from $R$ being a vertex of outdegree $j$.\\
\end{proof}

Fix $k \in \bN$. It follows that the generating function for all nonempty rooted plane trees with outdegrees divisible by $k$ is 
$$t_k(v,z_0,z_1,\ldots) = \sum \limits_{j=0}^{\infty} v^{jk} z_{jk} t(v,z_0,z_1,\ldots)^{jk}$$
To apply Lagrange's Implicit Function Theorem, define 
$$s_k(v,z_0,z_1,\ldots) := v t_k(v,z_0,z_1,\ldots) = v \bigg(z_0 + \sum_{j=1}^{\infty} z_{jk} s_k(v,z_0,z_1,\ldots)^{jk} \bigg) $$
But there is a natural bijection between $s_k(v,z_0,z_1,\ldots)$ and $g_k(u,w_1,\ldots)$ given by $v \mapsto u$, $z_0 \mapsto 1, z_j \mapsto w_j \ \forall j \geq 1$ (by observing that $z_0$ appearing in the above expression for $s_k$ stands for the tree with 1 node, which corresponds to the ``1" in the expression for $g_k$). In particular, in this bijection between $NC^k(n) \cup \{\emptyset\}$ and $T^k(n)$, edges of the tree correspond to elements of the partition and the outdegree of the nodes corresponds to the size of the blocks. This allows us to translate results involving trees to the context of non-crossing partitions, as done by Dershowitz and Zaks in \cite{trees}.

For example, the following is a rooted plane tree and its corresponding non-crossing partition. To get this, note that there are 12 edges so we want a non-crossing partition of $\{1,\ldots,12\}$; the root has outdegree 2 and has 3 edges in its right subtree so the block containing 1 is $\{1,9\}$; recurse on the left and right subtrees of the root (i.e. on $\{2,3,4,5,6,7,8\}$ and on $\{10,11,12\}$). 

\begin{center}
\begin{tikzpicture}
[level distance=10mm,
 my circle/.style={draw, fill, circle, minimum size=5pt, inner sep=0pt},
level 1/.style={sibling distance=30mm},
level 2/.style={sibling distance=15mm},
level 3/.style={sibling distance=10mm}]
\node[my circle]  {} [grow'=up]
child {node[my circle] {}
child {node[my circle] {}
child {node[my circle] {}}
child {node[my circle] {}}
}
child {node[my circle] {}}
child {node[my circle] {}
child {node[my circle] {}}
child {node[my circle] {}}
}
}
child {node[my circle] {}
child {node[my circle] {}
child[missing]
child {node[my circle] {}}
}
child {node[my circle] {}}};
\path (1.2,0.4) node (a) {9};
\path (-1.2,0.4) node (a) {1};
\path (2.2,1.4) node (a) {12};
\path (0.8,1.4) node (a) {10};
\path (-0.4,1.4) node (a) {6};
\path (-1.7,1.6) node (a) {5};
\path (-2.5,1.4) node (a) {2};
\path (1.3,2.4) node (a) {11};
\path (0.5,2.4) node (a) {8};
\path (-0.5,2.4) node (a) {7};
\path (-2.5,2.4) node (a) {4};
\path (-3.5,2.4) node (a) {3};
\end{tikzpicture}\\
\end{center}
\begin{center}
\begin{tikzpicture}[thick,font=\small]
     \path (0,0) node (a) {1}
           (1,0) node (b) {2}
           (2,0) node (c) {3}
           (3,0) node (d) {4}
           (4,0) node (e) {5}
           (5,0) node (f) {6}
           (6,0) node (g) {7}
           (7,0) node (h) {8}
           (8,0) node (i) {9}
           (9,0) node (j) {10}
           (10,0) node (k) {11}
           (11,0) node (l) {12};
     \draw (a) -- +(0,0.75) -| (i);
     \draw (j) -- +(0,0.75) -| (l);
     \draw (k) -- +(0,0.5) ;
     \draw (b) -- +(0,0.6) -| (e);
     \draw (e) -- +(0,0.6) -| (f);
     \draw (g) -- +(0,0.6) -| (h);
     \draw (c) -- +(0,0.4) -| (d);

   \end{tikzpicture}
\end{center}

Using this bijection, we get immediate corollaries of Theorems 4.3.5 and 4.3.7.\\

\noindent For any $n, k \in \bN$ consider the random variable $\widetilde{X}_n^{(k)}$ counting the number of edges of a $T\in T^k(n)$ (taken uniformly at random). 
\begin{corollary}[4.4.3]
 Fix any $k \in \bN$. Then $\frac{\widetilde{X}_n^{(k)} - \bE[\widetilde{X}_n^{(k)}]}{\sqrt{Var \ \widetilde{X}_n^{(k)}}} \xrightarrow{d} N(0,1).$
\end{corollary}

\noindent For any $n, k, t \in \bN$ with $ t \leq n$ consider the random variable $\widetilde{Y}_n^{(t,k)}$ counting the number of nodes with outdegree $tk$ in a $T \in T^k(n)$ (taken uniformly at random).  
\begin{corollary}[4.4.4]
 Fix any $t, k \in \bN$. Then $\frac{\widetilde{Y}_n^{(t,k)} - \bE[\widetilde{Y}_n^{(t,k)}]}{\sqrt{Var \ \widetilde{Y}_n^{(t,k)}}} \xrightarrow{d} N(0,1).$
\end{corollary}

Note that the random variables $Y_n^{(t,k)}, \widetilde{Y}_n^{(t,k)}$ are only considered for $t \geq 1$, since ``the number of blocks of size 0 in a non-crossing partition" is not well-defined. However, we can still consider the number of vertices in a rooted plane tree with outdegree 0 (i.e. the number of leaves), i.e. $\widetilde{Y}_n^{(0,k)}$. By this bijection and by Proposition 4.3.6 (with $t  =0$), we get the factorial moments of $\widetilde{Y}_n^{(0,k)}$. 

\begin{proposition}[4.4.5]
$$\bE[(\widetilde{Y}_n^{(0,k)})_p] = \frac{(kn+1)_p  (kn)_{p}} {(k n+ n)_{ p}} \ \forall n,k,p \in \bN \ \mbox{with} \ p \leq nk$$
\end{proposition}

\begin{corollary}[4.4.6]
 Fix any $k \in \bN$. Then $\frac{\widetilde{Y}_n^{(0,k)} - \bE[\widetilde{Y}_n^{(0,k)}]}{\sqrt{Var\ \widetilde{Y}_n^{(0,k)}}} \xrightarrow{d} N(0,1).$
\end{corollary}
\begin{proof}
We want to apply Corollary 3.3.1 with $u_p = v_p = 0 = \deg y_p = \deg z_p$ and with \ \ $\sigma = y_p(x) = z_p (x) = 1$. Note that (3) and (4) hold trivially. Next, we check condition (1) by proving $0 < $ LHS $-$ RHS of (1). We have
$$g_{k} := \mbox{LHS $-$ RHS of (1)} = \frac{1^2}k + \frac{1^2}k  - \frac{1^2}{k+1} = \frac{k + 2}{k(k+1)} > 0 $$
 Thus (1) holds. To check condition (2), we check that $0 < $ LHS $-$ RHS of (2). We have
$$h_{k} := \mbox{LHS $-$ RHS of (2)} = (k+1) - k^1 \cdot k^1 \cdot g_{k} = (k+1) - \frac{k(k+2)}{k+1} = \frac1{k+1} > 0$$
Therefore (2) holds. The desired conclusion follows by Corollary 3.3.1.\\
\end{proof}

There are various different bijections between rooted plane trees and non-crossing partitions which can be used to yield interesting correspondences between seemingly unrelated statistics. Moreover, this simple procedure of determining the generating function, using Lagrange inversion to compute the factorial moments, and applying Corollary 3.3.1 will work effectively in many combinatorial settings beyond non-crossing partitions and rooted plane trees.


\begin{thebibliography}{3}
\bibitem{arizmendi} 
O. Arizmendi. 
\textit{Statistics of blocks in k-divisible non-crossing partitions}.
2012,\\ (arXiv: 1201.6576).

\bibitem{billingsley} 
P. Billingsley.
\textit{ Probability and Measure}.
 Wiley-Interscience, 1995. 

\bibitem{cumulants} 
K. Bobecka, P. Hitczenko, F. L\'{o}pez-Bl\'{a}zquez, G. Rempa\l{}a and J. Weso\l{}owski.
\textit{ Asymptotic Normality through Factorial Cumulants and Partitions Identities}. 
Jan. 2013, \\ (arXiv: 1106.4684).

\bibitem{bona} 
M. B\'{o}na.
\textit{On Three Diﬀerent Notions of Monotone Subsequence}.
2007,\\ (arXiv: 0711.4325).

\bibitem{trees}
N. Dershowitz and S. Zaks.
Ordered Trees and Non-crossing Partitions.
\textit{Discrete Mathematics} 62(2) (1986): 215-218.

\bibitem{ancomb}
P. Flajolet and R. Sedgewick.
\textit{Analytic Combinatorics}.
Cambridge University Press, 2009.

\bibitem{gao}
Z. Gao and C. N. Wormald.
\textit{Asymptotic normality determined by high moments, and submap counts of random maps},
Probability Theory and Related Fields, vol. 130, no. 3, pp. 368 -- 376, 2004. 

\bibitem{janson}
S. Janson.
\textit{Normal convergence by higher semi-invariants with applications to sums of dependent random variables and random graphs}, Ann. Prob. 16, no. 1, pp. 305--312, 1988.

\bibitem{nica} 
A. Nica and R. Speicher. 
\textit{ Lectures on the Combinatorics of Free Probability}. 
London Mathematical Society Lecture Note Series, vol. 335, Cambridge University Press, 2006.

\bibitem{scheinerman} 
E. R. Scheinerman.
\textit{Mathematics: A Discrete Introduction}.
Google Books. Cengage Learning, 2012. Web. 
\texttt{goo.gl/whraD6}

\bibitem{weiss}
I. Weiss.
Limiting Distributions in Some Occupancy Problems.
\textit{The Annals of Mathematical Statistics} 29(3) (1958): 878-884.





\end{thebibliography}
\end{document}